\documentclass[12pt]{amsart}

\usepackage{amssymb}
\usepackage{mathrsfs}

\newtheorem{thm}{Theorem}[section]
\newtheorem{prop}[thm]{Proposition}
\newtheorem{claim}{Claim}
\newtheorem*{thm*}{Theorem}
\newtheorem*{claim*}{Claim}

\theoremstyle{remark}

\theoremstyle{definition}
\newtheorem{defn}[thm]{Definition}

\newcommand\Acal{\mathcal{A}}
\newcommand\Dcal{\mathcal{D}}
\newcommand\Tcal{\mathcal{T}}

\newcommand\Zcal{\mathcal{Z}}

\newcommand\Qbb{\mathbb{Q}}
\newcommand\Nbb{\mathbb{N}}
\newcommand\Zbb{\mathbb{Z}}
\newcommand{\Seq}[1]{\langle #1 \rangle}
\newcommand{\symdif}{\triangle}

\title[A note on Shavgulize's proof]
{A note on Shavgulidze's papers concerning the amenability problem for Thompson's group $F$}

\author{Justin Tatch Moore}

\address{Department of Mathematics \\ Cornell University\\
Ithaca, NY 14853--4201 \\ USA}

\email{{\tt justin@math.cornell.edu}}
 
\begin{document}

\maketitle

The purpose of this note is to examine E.
Shavgulidze's publications \cite{amen_disc_diff:abs}, \cite{Thompson_F_amen:Sh}
and ArXiv postings \cite{amen_diffeo:ArXiv}, \cite{amen_diffeo2:ArXiv} in which he claims to
have solved the amenability problem for Thompson's group $F$.
There is a consensus among those who have carefully studied these papers that not only do they
contain errors, but that the errors are serious and do not seem to be
repairable.

Unfortunately,  both \cite{amen_disc_diff:abs} and \cite{Thompson_F_amen:Sh}
have been published and are now being referred to
as though they represent a correct solution to the amenability problem for $F$.
This is further complicated by the fact that an inaccurate review of \cite{amen_disc_diff:abs} has appeared in
Math Reviews MR2486813.
Also, reviews of \cite{amen_disc_diff:abs} and \cite{Thompson_F_amen:Sh} on Zentralblatt MATH
treat these articles as correct.
Finally, to my knowledge, Shavgulidze has made no public acknowledgment that there are problems
with \cite{amen_disc_diff:abs} and \cite{Thompson_F_amen:Sh}.

Before proceeding, I will briefly describe the papers and preprints of Shavgulidze under discussion.
\cite{amen_disc_diff:abs} is an extended abstract  
in which Shavgulidze announces that he has proved a general
theorem which has, as a corollary, that Thompson's group $F$ is amenable.
Contrary to what is asserted in Math Reviews MR2486813, \cite{amen_disc_diff:abs} does not contain proofs
of the results it announces.
\cite{Thompson_F_amen:Sh} is a longer paper 
which contains the proof which is outlined in \cite{amen_disc_diff:abs}.
\cite{amen_diffeo:ArXiv} is a preprint of \cite{Thompson_F_amen:Sh} posted to the ArXiv.
\cite{amen_diffeo2:ArXiv} is a preprint posted to the ArXiv which apparently attempts to repair the problem
known to exist in \cite{Thompson_F_amen:Sh}, \cite{amen_diffeo:ArXiv}.
It should be noted however that \cite{amen_diffeo:ArXiv} and \cite{amen_diffeo2:ArXiv} are separate postings
to the ArXiv, not different versions of the same posting.
Furthermore, \cite{amen_diffeo2:ArXiv} makes no reference to any of the other papers mentioned above
and in particular does not indicate that there are errors in \cite{Thompson_F_amen:Sh},
\cite{amen_diffeo:ArXiv}.

The present note aims to both point to serious errors in these papers and to
argue more generally why the approach taken by these papers seems unlikely to yield
a solution to the amenability problem for $F$.
Much of what is contained in this note was circulated to a limited number of people at
Vanderbilt and SUNY Binghamton during Shavgulidze's visit to the US in January 2010.
It will be integrated into a broader survey article \cite{anal_amen_F}.
I would like to acknowledge the hard work of all those involved in the project of reading Shavgulidze's postings
to the ArXiv.
I became actively involved in the reading project at a relatively late stage and the notes from Matt Brin's seminar
\cite{Shav_notes:Brin} and private communication with Matt Brin, Victor Guba, and Mark Sapir were very helpful.
Much of what follows was precipitated by Matt Brin's talk in the Topology and Geometric Group Theory Seminar
at Cornell on 12/1/2009.

\section{Specific problems with the proofs}

There seems to be general agreement
that the problems with the proof in \cite{Thompson_F_amen:Sh} (and \cite{amen_diffeo:ArXiv}) are limited to
Lemma 2.4 (Lemma 5 of \cite{amen_diffeo:ArXiv})
which asserts that a certain sequence of Borel measures $u_n$ $(n \in \Nbb)$ satisfies a condition
which I will refer to as the \emph{mesh condition}.
There is agreement that the arguments of \cite{amen_diffeo:ArXiv} show that
the existence of a sequence of Borel measures $u_n$ $(n \in \Nbb)$ which satisfy the mesh condition and
additionally an \emph{invariance condition} (the conclusion of Lemma 6 of \cite{amen_diffeo:ArXiv})
are sufficient to establish the amenability of $F$.
It was observed by the original team of readers of \cite{amen_diffeo:ArXiv}
that the original proposed sequence of Borel
measures does \emph{not} in fact satisfy the mesh condition as was claimed in \cite{amen_diffeo:ArXiv}
(see pages 36-39 of \cite{Shav_notes:Brin}).
During Shavgulidze's visit to several US universities in January 2010,
he proposed a revised sequence of measures $u_n$ $(n \in \Nbb)$, this time with
finite support, claiming that they satisfied both the mesh and invariance conditions.
The details of this construction were limited at the time and were supplied much later in
\cite{amen_diffeo2:ArXiv}.

Section \ref{elem_proof}
contains an elementary proof (unlike the one provided in \cite{amen_diffeo2:ArXiv})
that the amenability of $F$ follows from the existence measures such as those constructed
in \cite{amen_diffeo2:ArXiv}.
This argument makes all but pages 10 and 11 of \cite{amen_diffeo2:ArXiv} irrelevant and eliminates the
need for the sophisticated analytical tools which Shavgulidze utilized in his proofs.
A similar but more elaborate argument can be used to show that the existence of a sequence of
Borel measures $u_n$ $(n \in \Nbb)$ satisfying the conclusions Lemmas 5 and 6 of \cite{amen_diffeo:ArXiv}
is equivalent to the amenability of $F$.

The paper \cite{amen_diffeo2:ArXiv} is full of typographical errors, minor mathematical errors,
and in general is poorly written at a basic mechanical level.
This in part makes it difficult to pinpoint the exact location of the real mistake.
Still, a serious error is contained in Lemma 5.
The proof concludes with an estimate concerning the sets $X^{0,l,n}$ but Lemma 5 concerns
the sets $X^{m,l,n}$ for arbitrary $m$.
No justification is given for this discrepancy.
What seems to be implicit is that the functions $\kappa_n$ commute with the
maps $f_1$ and $f_2$ (which are the generators of the group) and this is false.
It is also clear that the construction on pages 10-11 would result in a F\o lner sequence which would
violate the following theorems of \cite{fast_growth_F}:

\begin{thm} \cite{fast_growth_F} \label{tower_growth}
There is a constant $C$ such that if $A$ is a $C^{-n}$-F\o lner set in $F$ (with respect
to the standard generating set)
then $|A| \geq \exp_n(0)$, where
$\exp_0(m) = m$ and $\exp_{n+1}(m) = 2^{\exp_n (m)}$.
\end{thm}

\begin{thm} \cite{fast_growth_F} \label{monotone}
If $\mu$ is a finitely additive probability measure
on $\Tcal$ which is invariant with respect to the action of $F$,
then for every sequence $I_i$ $(i \leq k)$ of intervals satisfying
$0 < \min I_i < \max I_i < \min I_{i+1} < \max I_{i+1} < 1$ for all $i < k$, 
it follows that $\mu$-a.e. $T$ satisfies that the cardinalities
$|T \cap I_i|$ $( i \leq k)$ form a strictly monotonic sequence.
\end{thm}

\section{The Elementary proof}

\label{elem_proof}

I will first review some notation.
Let $\Dcal$ denote all finite subsets of $[0,1]$ which contain $\{0,1\}$.
Let $\Tcal$ denote the collection of all elements $T$ of $\Dcal$
such that, if $t_i$ $(i \leq k)$ is the increasing enumeration of $T$,
then for all $i < k$, there is are non-negative
integers $p$ and $q$ such that $t_i = p/2^q$ and $t_{i+1} = (p+1)/2^q$.
If $(S,T)$ is a pair of elements of $\Tcal$ of the same cardinality,
then the increasing map from $S$ to $T$
extends linearly on the complement to a piecewise
linear map from $[0,1]$ to $[0,1]$.
The collection of all such maps
under composition is one of the standard models of $F$ \cite{CFP}.
The standard generators for $F$ are the elements
\[
x_0 = (\{0,\frac{1}{2},\frac{3}{4},1\},\{0,\frac{1}{4},\frac{1}{2},1\})
\]
\[
x_1 = (\{0,\frac{1}{2},\frac{3}{4},\frac{7}{8},1\},\{0,\frac{1}{2},\frac{5}{8},\frac{3}{4},1\})
\]
In what follows, \emph{generator} will always refer $x_0$, $x_1$, $x_0^{-1}$, or $x_1^{-1}$.
Set
\[
I_n = \{1-2^{-i} : 0 \leq i \leq n+1\} \cup \{1\}.
\]
If $T$ is in $\Tcal$ and $|I_n| = |T|$, then I will use
$f_T$ to denote the element of $F$ represented by
$(I_n,T)$.

While the amenability of the partial action of $F$ on $\Tcal$ is equivalent to the amenability of
$F$ (see \cite{fast_growth_F} for a direct proof of this using similar notation to the present note),
the action of $F$ on $\Dcal$ is easily seen to be amenable.
In fact if $A \subseteq \Nbb$ is an $\epsilon$-F\o lner set in $\Zbb$, then it is not 
hard to see that
\[
\{ \{0,1-2^{-n-2},1\} : n \in A\}
\]
is a $2\epsilon$-F\o lner set with respect to the action of $F$ on $\Dcal$
(each generator has a translation error of $\epsilon$).
It turns out, however, that if one requires that the F\o lner sequence concentrates on sets
of mesh at most $1/16$, then one again obtains a reformulation of the amenability of $F$.

\begin{prop} \label{fin_supt}
$F$ is amenable provided that, for every $\epsilon$ there is a finite
$\Acal \subseteq \Dcal$ such that
\[
|\Acal \symdif (\Acal \cdot x_i)| < \epsilon |\Acal|
\]
for each $i < 2$ and for every $X$ in $\Acal$,
$||X|| \leq 1/16$.
\end{prop}



\begin{defn}
If $X$ is in $\Dcal$, define $T(X)$ to be the maximum element of $\Tcal$
such that if $s < t$ are in $T$, then there is an $x$ in $X$ with
$s \leq x < t$.
\end{defn}

\begin{claim}
If $T$ is in $\Tcal$ and
$g$ is represented by $(U,V)$ and
$U$ is contained in $T$, then $g \cdot T$ is in $\Tcal$.
Furthermore $g \circ f_T = f_{g \cdot T}$.
\end{claim}

\begin{proof}
See \cite{CFP}.
\end{proof}

\begin{claim}
If $X$ is in $\Dcal$ and
$f$ is represented by $(U,V)$, $U$ is contained in
$T(X)$, then $f \cdot T(X) = T(f \cdot X)$.
\end{claim}

\begin{proof}
Let $X$, $f$ be as above.
Since $f \cdot T(X)$ is in $\Tcal$,
it is contained in $T(f \cdot X)$
(this follows from the fact that $f$ is increasing).
This in turn implies $f^{-1} \cdot T(f \cdot X)$
is in $\Tcal$ and therefore is contained in
$T(X)$.
Hence $|T(X)| = |f \cdot T(X)| = |T(f \cdot X)|$
and therefore we must have
$f \cdot T(X) = T(f \cdot X)$.
\end{proof}

\begin{claim}
If $||X|| \leq 1/16$, then $||T(X)|| \leq 1/8$.
In particular $I_0$ and $I_1$ are contained
in $T(X)$.
\end{claim}

\begin{proof}
Suppose $X$ is as above and that $s < t$ are consecutive
elements of $T(X)$.
Since $T(X)$ is maximal, either
there is no $x$ is $X$ such that $s \leq x < (s+t)/2$ or
else there is no $x$ in $X$ such that
$(s+t)/2 \leq x < t$.
Hence there is an interval of length
$(t-s)/2$ contained in $[0,1]$ and disjoint from $X$.
It follows that $(t-s)/2 \leq 1/16$ and therefore
that $t-s \leq 1/8$.
\end{proof}

Putting this together, we have that $i = 0,1$ and
$\Zcal \subseteq \Dcal$ consists only of $Z$ such that
$|Z| < 1/8$, then
\[
\{f_{T(X)} : X \in (x_i \cdot \Zcal)\} = 
\{f_{T(x_i \cdot Z)} : Z \in \Zcal\}
\]
\[
= \{f_{x_i \cdot T(Z)} : Z \in \Zcal\}
= \{x_i \circ f_{T(Z)} : Z \in \Zcal\}
\]
In particular, if additionally $\Zcal$ is finite and
\[
|(\Zcal \cdot x_i) \symdif \Zcal| < \epsilon |\Zcal|
\]
and $\Acal = \{f_{T(Z)} : Z \in \Zcal\}$,
then $|(x_i \circ \Acal) \symdif \Acal| < \epsilon |\Acal|$.
This finishes the proof of Proposition \ref{fin_supt}.

\def\Dbar{\leavevmode\lower.6ex\hbox to 0pt{\hskip-.23ex \accent"16\hss}D}


\begin{thebibliography}{1}

\bibitem{Shav_notes:Brin}
Matthew Brin.
\newblock On {S}havgulidze's proof the amenability of some discrete groups of
  homeomorphisms of the unit interval.
\newblock ArXiv preprint 0908.1353v5, Sept. 2009.

\bibitem{CFP}
J.~W. Cannon, W.~J. Floyd, and W.~R. Parry.
\newblock Introductory notes on {R}ichard {T}hompson's groups.
\newblock {\em Enseign. Math. (2)}, 42(3-4):215--256, 1996.

\bibitem{fast_growth_F}
Justin~Tatch Moore.
\newblock Fast growth in {F}\o lner function for {T}hompson's group $f$.
\newblock ArXiv Preprint 0905.1118, Aug. 2009.

\bibitem{anal_amen_F}
Justin~Tatch Moore.
\newblock Analysis of the amenability problem for {T}hompson's group ${F}$.
\newblock in preparation, February 2011.

\bibitem{amen_disc_diff:abs}
E.~T. Shavgulidze.
\newblock Amenability of discrete subgroups of the group of diffeomorphisms of
  the circle.
\newblock {\em Russ. J. Math. Phys.}, 16(1):130--132, 2009.

\bibitem{Thompson_F_amen:Sh}
E.~T. Shavgulidze.
\newblock The {T}hompson group {$F$} is amenable.
\newblock {\em Infin. Dimens. Anal. Quantum Probab. Relat. Top.},
  12(2):173--191, 2009.

\bibitem{amen_diffeo:ArXiv}
E.T. Shavgulidze.
\newblock About amenability of subgroups of the group of diffeomorphisms of the
  interval.
\newblock ArXiv Preprint 0906.0107, May 30 2009.

\bibitem{amen_diffeo2:ArXiv}
E.T. Shavgulidze.
\newblock About amenability of subgroups of the group of diffeomorphisms of the
  interval $[0,1]$.
\newblock ArXiv Preprint 1101.2888, January 14 2011.

\end{thebibliography}
\end{document}